\documentclass[a4paper,12pt]{article}
\usepackage{graphicx,amssymb,amstext,amsmath}

\usepackage[latin1]{inputenc}
\usepackage{amsthm}
\usepackage{dsfont}
\usepackage{amsmath}
\usepackage{amsfonts}
\usepackage{amssymb}
\usepackage{mathrsfs}
\usepackage{graphicx}
\usepackage{graphicx,color}
\usepackage[all]{xy}
\usepackage[total={5.2in,9.1in},
top=1.4in, left=1.55in, includefoot]{geometry}

\newtheorem{R}{Remark}

\newtheorem{Pro}{Proposition}
\newtheorem{C}{Corollary}

\newtheorem{Le}{Lemma}

\newcommand{\ds}{\displaystyle}

\newcommand{\re}{\mathbb{R}}

\newcommand{\tb}{\theta_{w}^{\bot}}
\newcommand{\ab}{a_{w}^{\bot}}

\title{The Marstrand Theorem in Nonpositive Curvature}
\author{Sergio Augusto Roma\~na Ibarra\footnote{Partially supported by CNPq, Capes and the Palis Balzan Prize.}}
\begin{document}
\maketitle
\begin{abstract}
\noindent In a paper from 1954, Marstrand proved that if $K\subset \re^2$ with
Hausdorff dimension greater than 1, then its one-dimensional projection has positive
Lebesgue measure for almost-all directions. In this article, we show that if $M$ is a simply connected surface with non-positive curvature, then Marstrand's theorem is still valid.

\end{abstract}
\section{Introduction}\label{IntroMars}
Consider $\re^2$ as a metric space with a metric $d$. If $U$ is a subset of $\re^2$, the diameter of $U$ is $|U|=\sup\{d(x,y):x,y\in U\}$ and, if $\mathcal{U}$ is a family of subsets of $\re^{2}$, the diameter of $\mathcal{U}$ is defined by 
$$\left\|\mathcal{U}\right\|=\sup_{U\in\ \mathcal{U}}|U|.$$
Given $s>0$, the Hausdorff $s$-measure of a subset $K$ of $\re^2$ is 
$$m_s(K)=\lim_{\epsilon \to 0}\left( \inf_{ \stackrel{\mathcal{U}\ \text{covers} \ K}{\|U\|<\epsilon}}\sum_{U\in \ \mathcal{U}}{|U|}^{s} \right).$$
In particular, when $d$ is the Euclidean metric and $s=1$, then $m=m_1$ is the Lebesgue measure. It is not difficult to show that there exists a unique $s_0\geq 0$ for  which $m_s(K)=+\infty$ if $s<s_0$ and $m_s(K)=0$ if $s>s_0$. We define the Hausdorff dimension of $K$ as $HD(K)=s_0$. Also, for each $\theta\in \re$, let $v_\theta=(\cos \theta,\sin \theta)$, $L_{\theta}$ the line in $\re^2$  through of the origin containing $v_{\theta}$ and $\pi_{\theta}:\re^2\to L_{\theta}$ the orthogonal projection.\\
In 1954,  J. M. Marstrand \cite{Marst} proved the following result on the fractal dimension
of plane sets.\\
\ \\
\textbf{Theorem[Marstrand]:}\textit{\ If $K\subset \re^2$ such that $HD(K)>1$, then
$m(\pi_{\theta}(K))>0$ for m-almost every $\theta\in \re$.}\\
\ \\
The proof is based on a qualitative characterization of the ``bad" angles $\theta$ for
which the result is not true.\\
\ \\
Many generalizations and simpler proofs have appeared since. One of them
came in 1968 by R. Kaufman, who gave a very short proof of Marstrand's
Theorem using methods of potential theory. See \cite{Kaufman} for his original proof and
\cite{PT}, \cite{Falconer} for further discussion. Another recent proof of the theorem (2011), which uses combinatorial techniques is found in \cite{YG}.\\
\ \\
In this article, we consider $M$ a  simply connected surface with a Riemannian metric of non-positive curvature, and using the potential theory techniques of Kaufman \cite{Kaufman}, we show the following more general version of the Marstrand's Theorem.\\
\ \\
\noindent \emph{\bf The Geometric Marstrand Theorem:} \ \ \textit{Let $M$ be a Hadamard surface, let $K\subset M$ and $p\in M$, such that $HD(K)>1$, then for almost  every line $l$ coming from p, we have $\pi_{l}(K)$ has positive Lebesgue measure, where $\pi_l$ is the orthogonal projection on $l$.}\\
\ \\
\noindent Then using the Hadamard's theorem (cf. \cite{ManP}), the theorem above can be stated as follows:\\

\noindent \emph{\bf Main Theorem:}\ \ \textit{Let ${\mathbb R}^2$ be endow with a metric $g$ of non-positive curvature, and $K \subset {\mathbb R}^2$ with $HD(K)> 1$. Then for almost every $\theta \in (-\pi/2 , \pi /2)$, we have that $m(\pi_\theta(K))> 0$, \newline where $\pi_{\theta}$ is the orthogonal projection with the metric $g$ on the line $l_{\theta}$, of initial velocity $v_{\theta}=(\cos\theta,\sin\theta)\in T_{p}\re^{2}$.}

\section{Preliminaries} \label{Preliminaries}

\noindent Let $M$ be a Riemannian manifold with metric $\left\langle \ , \ \right\rangle$, a line in $M$ is a geodesic defined for all parameter values and minimizing distance between any of its points, that is, $\gamma : {\mathbb R} \to M$ is a isometry. 
If $M$ is a manifold of dimension $n$, simply connected and non-positive curvature, then the space of lines leaving  of a point $p$ can be seen as a sphere of dimension $n-1$. So, in the case of surfaces the set of lines agrees with $S^1$ in the space tangent $T_{p}M$ of the point $p$. Therefore, in each point on the surface the set of lines can be oriented and parameterized by $\left(-\frac{\pi}{2} , \frac{\pi}{2}\right]$ and endowed of Lebesgue measure. Thus, using the previous identification, we can talk about almost every line through a point of $M$ (cf. \cite{Bridson}).
In the conditions above, Hadamard's theorem states that $M$ is diffeomorphic to $\re^{n}$, (cf. \cite{ManP}).\\
\ \\
Moreover, given a geodesic triangle $\Delta ABC$ with sides , $\vec{BC}$ and $\vec{AC}$ denote by $\angle A$ the angle between geodesic segments $\vec{AB}$ and $\vec{AC}$, then \textit{the law of cosines} says 
$$|\vec{BC}|^{2}\geq|\vec{AB}|^{2}+|\vec{AC}|^2-2|\vec{AB}||\vec{AC}|\cos\angle A,$$
where $|\vec{ij}|$ is the distance between the points $i,j$ for $i,j\in \{A,B,C\}$.\\

\noindent \textbf{Gauss's Lemma:} Let $p\in M$ and let $v,w\in B_{\epsilon}(0)\in T_vT_p M \approx T_p M$ and $M\ni q=exp_pv$. Then, 
$$\left\langle d(exp_p)_{v}v,d(exp_p)_{v}w\right\rangle_{q}=\left\langle v,w\right\rangle_{p}.$$

\subsection{Projections}\label{Projections}
Let $M$ be a manifold simply connected and of non-positive curvature. Let $C$ be a complete convex set in $M$. \textit{The orthogonal projection} (or simply 
`\textit{projection}') is the name given to the map $\pi\colon M\to C$ constructed in the following proposition: (cf. \cite[pp 176]{Bridson}).
\begin{Pro}The projection $\pi$ satisfies the following properties:
\begin{enumerate}
\item For any $x\in M$ there is a unique point $\pi(x)\in C$ such that $d(x,\pi(x))=d(x, C) = \inf_{y\in C} d(x,y)$.
\item If $x_0$ is in the geodesic segment $[x, \pi(x)]$, then  $\pi(x_0)= \pi(x)$.
\item Given  $x \notin C$, $y \in C$ and  $y \neq \pi(x)$, then $\angle_{\pi(x)}(x,y) \geq \frac{\pi}{2}$.
\item $x \longmapsto \pi(x)$ is a retraction on $C$.
\end{enumerate}
\end{Pro}
\begin{C}
Let $M$, $C$ be as above and define  $d_C(x):=d(x,C)$, then 
\begin{enumerate}
\item $d_C$ is a convex function, that is, if $\alpha (t)$ is a geodesic parametrized proportionally to arc length, then
 $$d_C(\alpha(t)) \leq (1-t)d_C (\alpha(0)) + td_C \alpha(1)\ \ \text{for} \ \ t\in [0,1].$$
\item For all  $x,y \in M$, we have  $\left|d_C(x)- d_C(y)\right|\leq d(x,y)$.
\item The  restriction of $d_C$ to the sphere of center $x$ and radius $r\leq d_C(x)$ reaches the infimum in a unique point $y$ with
$$d_C(x)=d_C(y) + r.$$
\end{enumerate}
\end{C}
\noindent Here we consider $\re^{2}$ with a Riemannian metric $g$, such that the curvature $K_{\re^{2}}$ is non-positive, \emph{i.e.}, $K_{\re^{2}}\leq 0$.
\noindent Recall that a line $\gamma$ in ${\re}^2$ is a geodesic defined for all parameter values and minimizing distance between any of its points, that is,
$\gamma \colon {\re} \to {\re}^2$ and $d(\gamma(t),\gamma(s))=|t-s|$, where $d$ is the distance induced by the Riemannian metric $g$, in other words, a parametrization of  $\gamma$ is a isometry.
Then, given  $x\in \mathbb{R}^2$ there is a unique $\gamma(t_x)$ such that  $\pi_\gamma (x)=\gamma(t_x)$, thus without loss of generality we may call $\pi_\gamma(x)=t_x$.\\
\ \\
Fix $p \in \re^{2}$ and let $\{e_1,e_2\}$ be a positive orthogonal basis of $T_p\re^2$, \emph{i.e.}, the basis $\{e_1,e_2\}$ has the induced orientation of $\re^{2}$. Then, call $v_t=(\cos t,\sin t)$ in coordinates the unit vector $(\cos t) e_1+(\sin t) e_2\in T_{p}\re^{2}$. Denote by $l_t$ the line through $p$ with velocity $v_t$, given by $l_t(s)=exp_{p}sv_t$ and by $\pi_t$ the projection on $l_t$. 
Then, given $\theta\in [0,2\pi)$, we can define $\pi\colon [0,2\pi)\times T_{p}\re^2 \to \re$ by the unique parameter $s$ such that $\pi_{\theta}(exp_{p}w)=\exp_{p}s{v_{\theta}}$ \emph{i.e.}, $\pi(\theta,w):=\pi_{\theta}(w)$ and
$$\pi_{\theta}(exp_{p}w)=exp_{p}\pi(\theta, w)v_{\theta}.$$
 
\section{Behavior of the Projection $\pi$}\label{BP}
In this section we will prove some lemmas that will help to understand the projection $\pi$.
\subsection{Differentiability of $\pi$ in $\theta$ and $w$}
\begin{Le}\label{L1M}
The projection $\pi$ is differentiable in $\theta$ and \ $w$.
\end{Le}
\begin{proof}[\bf{Proof}]
Fix $w$ and call $q=exp_{p}{w}$. Let $\alpha_v(t)\subset T_q{\mathbb R}^2$ such that $exp_{q}\alpha_{v}(t)=\gamma_v(t)$, where $\gamma_v$ is the line  such that $\gamma_v(0)=p$ and $\gamma_v'(0)=v$, then, for all $v\in S^1$, there is a unique $t_v$ such that $d(q,\gamma_v({\mathbb R}))=d(q,\gamma_v(t_v))$ and satisfies
$$\left\langle d(exp_q)_{\alpha_v{(t_v)}} (\alpha_v'(t_v)) ,  d(exp_q)_{\alpha_v{(t_v)}} (\alpha_v(t_v))\right\rangle=
\left\langle \gamma_v'(t_v),d(exp_q)_{\alpha_v{(t_v)}} (\alpha_v{(t_v)})\right\rangle=0.$$
By Gauss Lemma, we have 
$$\frac{1}{2}\frac{\partial}{\partial t}\left\|\alpha_v(t)\right\|^2(t_v)= \left\langle \alpha_v'(t_v),\alpha_v(t_v) \right \rangle=0.$$
We define the real function 
\begin{eqnarray*}
\eta: S^1 &\times & {\mathbb R} \longrightarrow {\mathbb R}\\
\eta(v,t)&=&\frac{1}{2}\frac{\partial}{\partial t}\left\|\alpha_v(t)\right\|^2,
\end{eqnarray*}
\noindent this function is $C^\infty$ and satisfies $\eta(v,t_{v})=0$, also $\frac{\partial}{\partial t} \eta(v,t)= \frac{1}{2}\frac{\partial^2}{\partial t^2} \left\|\alpha_v(t)\right\|^2$.\\
\ \\
\noindent Put $g(t)={\left\| \alpha_{v_{0}}(t)\right\|}^2$, then $\frac{\partial}{\partial t} \eta(v_0,t_0)=\frac{1}{2}g''(t_0)$.
Also, $g(t)={d(q,\gamma_{v_0}(t))}^2$ is differentiable and has a global minimum at $t_{v_0}$, as $K_{\re^2}\leq 0$, $g$ is convex. In fact, for $s\in[0,1]$
\begin{eqnarray*}
\ds g(sx+(1-s)y)&=&d(q,\gamma_{v_0}(sx+(1-s)y))^2 \leq 
\left(sd(q,\gamma_{v_0}(x))+(1-s)d(q,\gamma_{v_0}(y))\right)^2 \\
&\leq& sd(q,\gamma_{v_0}(x))^2+(1-s)d(q,\gamma_v(y))^2=sg(x)+(1-s)g(y)
\end{eqnarray*}

\noindent by the law of cosines and using the fact $\angle_{\pi_{\gamma_{v_0}(t_0)}}(q,\gamma_{v_0}$(t)$) = \frac{\pi}{2}$  at the point of projection
\begin{center}
$d(q,\gamma_{v_0}(t_{v_0}))^2 + d(\gamma_{v_0}(t_{v_0}),\gamma_{v_0}(t))^2 \leq d(q,\gamma_{v_0}(t))^2$
\end{center}
equivalently
\begin{center}
$g(t_{v_0})+(t-t_{v_0})^2 \leq g(t)$.
\end{center}
Therefore, as $g'(t_{v_0})=0$, then $g''(t_{v_0}) > 0$.
This implies that $\frac{\partial \eta}{\partial t}(v_0,t_0)\neq 0$ and by Theorem of Implicit Functions, there is an open $U$ containing $(v_0,t_{v_0})$, a open $V \subset S^1$ containing $v_0$ and $\xi:V \longrightarrow {\mathbb R}$, a class function $C^\infty$ with $\xi(v_0)=t_{v_0}$ suct that 
$$\left\{(v,t)\in U: \eta(v,t)=0\right\} \Longleftrightarrow \left\{ v\in V : t=\xi(v)\right\}.$$
Since by construction $\eta(v,\xi(v))=0$ implies $\pi(v,q)= \xi(v)$, and therefore $\pi(v,q)$ is  diffe-\\rentiable in $v$, in fact it is $C^\infty$. The above shows that $\pi$ is differentiable in $\theta$.\\
\ \\
\noindent Analogously, is proven that $\pi$ is differentiable in $w$.

\end{proof}

\noindent Let $w\in T_{p}\re^2\setminus\{0\}$ and put $\theta_{w}^{\bot}\in [0,2\pi)$ such that $w$ and $v_{\theta_{w}^{\bot}}$ are orthogonal, that is $\left\langle w , v_{\theta_{w}^{\bot}}\right\rangle=0$, where the $\left\langle \ , \ \right\rangle$ is the inner product in $T_{p}\re^{2}$ and the set $\{w,v_{\theta_{w}^{\bot}}\}$ is a positive basis of $T_{p}\re^2$.
\begin{Le}\label{L2M}
The projection $\pi$ satisfies, $$\ds\frac{\partial \pi}{\partial \theta}\left(\theta_{w}^{\bot},w\right)=-\left\|w\right\|.$$
Moreover, there exists $\epsilon>0$ such that, for all $w$
$$-\left\|w\right\|\leq\frac{\partial \pi}{\partial \theta}\left(\theta,w\right)\leq -\frac{1}{2}\left\|w\right\| \ \text{and} \ \ \left|\frac{\partial^{2} \pi}{\partial^{2} \theta}\left(\theta,w\right)\right|\leq \left\|w\right\|,$$ whenever $\left|\theta-\theta_{w}^{\bot}\right|<\epsilon$.
\end{Le}
\noindent Before proving Lemma \ref{L2M} we will seek to understand the function $\pi(\theta,w)$.\\
\ \\
\noindent Let $\pi_{l_\theta}$ be the orthogonal projection on the line $\l_\theta$ generated by the vector $v_{\theta}$ in $T_{p}\re^{2}$,
in this case, $\pi_{l_\theta}(w)=\left\|w\right\|\cos(arg(w)-\theta)$, where $arg(w)$ is the argument of $w$ with relation to $e_1$ and the positivity of basis $\{e_1,e_2\}$.
\ \\ 
Now using the law of cosines
\begin{center}
$d(p,\pi_\theta(exp_{p} w))^2+d(exp_{p}w,\pi_\theta(exp_{p} w))^2\leq \left\|w\right\|^2=\pi_{l_\theta}(w)^2+d(\pi_{l_{\theta}}(w)v_\theta,w)^2,$
\end{center}

\noindent Since, $K\leq 0$, then $$d(exp_{p}w,\pi_\theta(exp_{p} w))=d(exp_{p}w,exp_{p}\pi_{\theta}(w)v_{\theta})\geq d(w,\pi_{\theta}(w)v_{\theta})\geq d(w,\pi_{l_{\theta}}(w)v_{\theta}).$$
Joining the previous expressions we obtain  
$$d(p,\pi_\theta(exp_{p}w))^2\leq \pi_{l_\theta}(w)^2 \Longleftrightarrow \pi_{\theta}(w)^{2}\leq \pi_{l_{\theta}}(w)^2.$$
\noindent Thus, since $\pi_{\theta}(w)$ has the same sign as $\pi_{l_{\theta}}(w)$, then 
\begin{eqnarray}
 \pi_\theta(w)\geq 0 &\Longrightarrow & \pi_{\theta}(w)\leq \pi_{l_{\theta}}(w)=\left\|w\right\|\cos(arg(w)-\theta); \label{E1}\\
 \pi_{\theta}(w)\leq 0  &\Longrightarrow& \pi_{\theta}(w)\geq \pi_{l_{\theta}}(w)=\left\|w\right\|\cos(arg(w)-\theta). \label{E2}
\end{eqnarray}

\begin{proof}[\bf{Proof of Lemma \ref{L2M}}] \ \\
As $\left\langle w,\theta_{w}^{\bot}\right\rangle=0$, then $arg(w)-\theta_{w}^{\bot}=-\pi/2$, thus $\pi(\theta_{w}^{\bot},w)=0=\left\|w\right\|\cos(-\pi/2)$. Moreover, as $\pi(\theta_{w}^{\bot}-h,w)\geq 0$ and $\pi(\theta_{w}^{\bot}+h,w)\leq 0$ for $h>0$ small, then 
$$\frac{\pi(\theta_{w}^{\bot}-h,w)}{h}\leq \frac{\left\|w\right\|\cos\left(arg(w)-(\theta_{w}^{\bot}-h)\right)}{h}$$
and 
$$\frac{\pi(\theta_{w}^{\bot}+h,w)}{h}\geq \frac{\left\|w\right\|\cos\left(arg(w)-(\theta_{w}^{\bot}+h)\right)}{h}.$$
If $h\to 0$ in the two previous inequalities we have
$$\frac{\partial \pi}{\partial \theta}(\theta_{w}^{\bot},w)\leq -\left\|w\right\|\sin(arg(w)-\theta_{w}^{\bot})=-\left\|w\right\|$$
and 
$$\frac{\partial \pi}{\partial \theta}(\theta_{w}^{\bot},w)\geq -\left\|w\right\|\sin(arg(w)-\theta_{w}^{\bot})=-\left\|w\right\|.$$
Therefore, 
\begin{equation}\label{E3}
\frac{\partial \pi}{\partial \theta}(\theta_{w}^{\bot},w)=-\left\|w\right\|.
\end{equation}

\noindent Moreover, for $h>0$ small and by the equation (\ref{E2}), we have 
\begin{eqnarray*}
\pi(\theta_{w}^{\bot}+h,w)&=&\frac{\partial \pi}{\partial \theta}(\theta_{w}^{\bot},w)h+\frac{1}{2}\frac{\partial^{2} \pi}{\partial^{2} \theta}(\theta_{w}^{\bot},w)h^{2}+r(h)\\
&\geq&\left\|w\right\|\cos\left(arg(w)-(\theta_{w}^{\bot}+h)\right)\\
&=&\left\|w\right\|\left(\frac{\partial}{\partial \theta} \cos(\theta-arg(w))|_{\tb}h+\frac{1}{2}\frac{\partial^{2}}{\partial^{2} \theta} \cos(\theta-arg(w))|_{\tb}h^{2}+R(h)\right).
\end{eqnarray*}
The above inequality and equation (\ref{E3}) implies that $$\frac{\partial^{2} \pi}{\partial^{2} \theta}(\theta_{w}^{\bot},w)h^{2}+r(h)\geq \left\|w\right\|\left(\frac{\partial^{2}}{\partial^{2} \theta} \cos(\theta-arg(w))|_{\tb}h^{2}+R(h)\right).$$
Since $\dfrac{\partial^{2}}{\partial^{2} \theta} \cos(\theta-arg(w))|_{\tb}=0$, then $\dfrac{\partial^{2} \pi}{\partial^{2} \theta}(\theta_{w}^{\bot},w)\geq 0$.
Analogously, using $\pi(\theta_{w}^{\bot}-h,w)$ and equation (\ref{E1}) we have that $\dfrac{\partial^{2} \pi}{\partial^{2} \theta}(\theta_{w}^{\bot},w)\leq 0$. So,
\begin{equation}\label{E4}
\dfrac{\partial^{2} \pi}{\partial^{2} \theta}(\theta_{w}^{\bot},w)=0.
\end{equation}

\noindent Using Taylor's expansion of third order for $\pi(\theta_{w}^{\bot}+h,w)$ and $h>0$, the equations (\ref{E2}), (\ref{E4}), and the fact that 
 $\dfrac{\partial^{3}}{\partial^{3} \theta} \cos(\theta-arg(w))|_{\tb}=1$, implies that 
$$\frac{\partial^{3} \pi}{\partial^{3} \theta}(\theta_{w}^{\bot},w)\frac{h^3}{6}+r_{3}(h)\geq \frac{h^3}{6}+R_{3}(h).$$
Thus,
\begin{equation}\label{E5}
\frac{\partial^{3} \pi}{\partial^{3} \theta}(\theta_{w}^{\bot},w)\geq 1.
\end{equation}

\noindent Equations (\ref{E4}) and (\ref{E5}) implies that, for any $w\in T_{p}\re^2$, the function $\dfrac{\partial \pi}{\partial \theta}(\cdot,w)$ has a minimum in $\theta=\tb$, therefore there is $\epsilon_{1}>0$ such that 
\begin{equation}\label{E6}
-\left\|w\right\|\leq \frac{\partial \pi}{\partial \theta}(\theta,w) \ \ \text{for all} \ \ \left|\theta-\tb\right|<\epsilon_1.
\end{equation}
\noindent The lemma will be proved if we show the following statements:

\begin{enumerate}
		\item There is $\delta_1>0$, such that for all $\left\|w\right\|\geq 1$, $$\frac{\partial \pi}{\partial \theta}(\theta,w)\leq -\frac{1}{2}\left\|w\right\|, \ \ \text{whenever} \left|\theta-\tb\right|<\delta_1.$$
		In fact: Let $1/2>\beta>0$, then by continuity of $\dfrac{\partial \pi}{\partial \theta}$, there is $\delta_1$ such that $$\text{if} \ \ \left|\theta-\tb\right|<\delta_1, \ \ \text{then}\ \  \frac{\partial \pi}{\partial \theta}(\theta,w)-\frac{\partial \pi}{\partial \theta}(\tb,w)<\beta.$$
		Thus, $\dfrac{\partial \pi}{\partial \theta}(\theta,w)<\beta-\left\|w\right\|<-\frac{1}{2}\left\|w\right\|$ for any $\left\|w\right\|\geq 1$.
	  \item There is $\epsilon_2>0$, such that for all $\left\|w\right\|=1$ and $t\in[0,1]$
		$$\frac{\partial \pi}{\partial \theta}(\theta,tw)\leq -\frac{1}{2}t, \ \ \text{whenever}\ \ \left|\theta-\tb\right|<\epsilon_2.$$
		In fact: Suppose by contradiction that for all $n\in \mathbb{N}$, there are $w_n, t_n, \theta_n$, $\left\|w_n\right\|=1$ such that $\left|\theta_{w_n}^{\bot}-\theta_n\right|<\frac{1}{n}$ and $\frac{\partial \pi}{\partial \theta}(\theta_n,t_n w_n)>-\frac{1}{2}t_n$. Without loss of generality, we can assume that $w_n\to w$, $\theta_n \to \tb$ and $t_n\to t$. If $t\neq 0$, the above implies a contradiction with (\ref{E3}). Thus, suppose that $t=0$, then 
		consider the $C^{1}$-function $H(\theta,t,w)=\frac{\partial \pi}{\partial \theta}(\theta,tw)$, then $\frac{\partial H}{\partial t}(\tb,0,w)=-\left\|w\right\|=-1$. Since $H$ is $C^1$, then  
		$$\lim_{n\to \infty}\frac{H(\theta_n,t_n,w_n)}{t_n}=\lim_{t\to 0}\frac{H(\theta,t,w)}{t}=-1<-1/2\leq \lim_{n\to \infty}\frac{H(\theta_n,t_n,w_n)}{t_n}.$$
Which is absurd, so the assertion 2 is proved.	
			
\end{enumerate}

\noindent Take $\epsilon=\min\{\epsilon_1,\epsilon_2,\delta_1\}$, then by the equation (\ref{E6}) and the statements 1 and 2 we have the second part of Lemma \ref{L2M}. The third part is analogous, just consider that $\frac{\partial ^{2}\pi}{\partial^2 \theta}(\tb,w)=0$. So we conclude the proof of Lemma.
\end{proof}
\begin{Le}\label{L3M}
Let $w\neq 0$ and $\theta \neq \tb$, then $\ds\lim_{t \to 0^{+}}\frac{\pi_{\theta}(tw)}{t}\neq 0$.
\end{Le}
\begin{proof}[\bf{proof}]
Suppose that $\lim_{t \to 0^{+}}\frac{\pi_{\theta}(tw)}{t}=0$, put $w(t)=exp_{p}tw$, let $v(t)\in T_{w(t)}\re^2$ the unit vector such that $exp_{w(t)}s(t)v(t)=\pi_{\theta}(exp_{p}tw)$ for some $s(t)\geq 0$. Let $J(t)\in T_{w(t)}\re^2$ such that $exp_{w(t)}J(t)=p$, that is, $J(t)=-d(exp_{p})_{tw}w$. Then, putting $\alpha(t)$ the oriented angle between $v(t)$ and $J(t)$ (cf. Figure \ref{fig:MF1}).\\

\begin{figure}[htbp]
	\centering
		\includegraphics[width=0.40\textwidth]{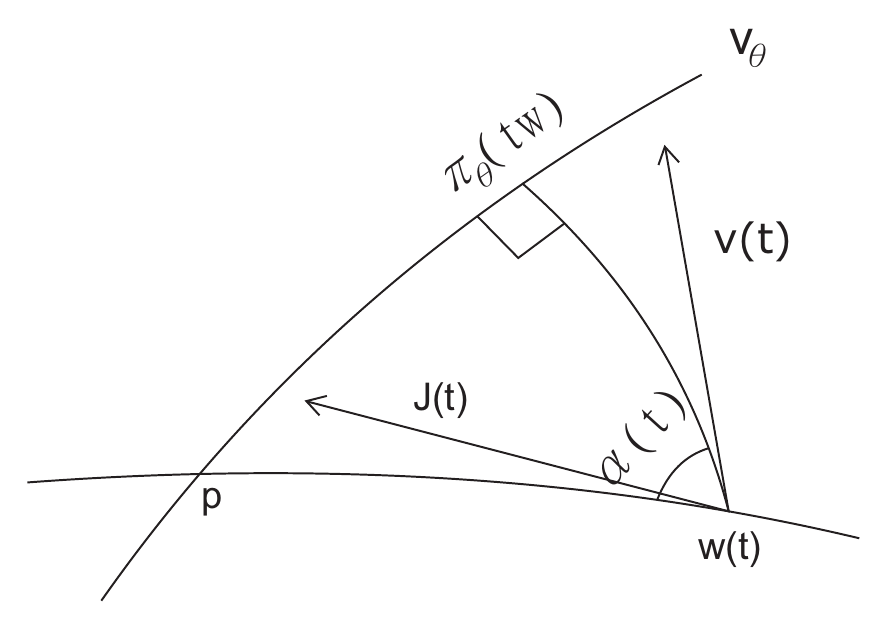}
	\caption{Convergence geodesics}
	\label{fig:MF1}
\end{figure}

\noindent By the law of cosines and using that $d(p,w(t))=\left\|J(t)\right\|=t\left\|w\right\|$ for $t>0$, and $\pi_{\theta}(tw)=d(p,\pi_{\theta}(w(t)))$, we obtain
$$\pi_{\theta}(tw)^{2}\geq \left\|J(t)\right\|^{2}+d(w(t),\pi_{\theta}(w(t)))^{2}-2\left\|J(t)\right\|d(w(t),\pi_{\theta}(w(t)))\cos \alpha(t).$$
Put $\ds\lim_{t\to 0^{+}}\frac{d(w(t),\pi_{\theta}(w(t)))}{t}=B$, then dividing by $t^{2}$ and when $t\to 0$ we have
\begin{eqnarray*}
0=\left(\lim_{t \to 0^{+}}\frac{\pi_{\theta}(tw)}{t}\right)^{2}&\geq& \left\|w\right\|^{2}+B^{2}-2\left\|w\right\|B\lim_{t\to 0^{+}}\cos \alpha(t)\\
&\geq &\left\|w\right\|^{2}+B^{2}-2\left\|w\right\|B=(\left\|w\right\|-B)^{2}\geq 0.\\
\end{eqnarray*}
Thus, we conclude that $B=\left\|w\right\|$ and $\ds\lim_{t\to 0^{+}}\cos \alpha(t)=1$. Therefore, $\ds\lim_{t\to 0^{+}}\alpha(t)=0$,  this implies the following  geodesic convergence 
$$exp_{w(t)}sv(t)\stackrel{t\to 0^{+}}{\longrightarrow} exp_{p}{s\frac{-w}{\left\|w\right\|}},$$
given that $w(t)\to p$ and $v(t)\to -\frac{w}{\left\|w\right\|}$ when $t\to 0^{+}$.\\
Moreover, by definition of $s(t)$, we have that 
$$\left\langle d\left(exp_{w(t)}\right)_{s(t)v(t)}v(t),d\left(exp_{p}\right)_{\pi_{\theta}(tw)v_{\theta}}v_{\theta})\right\rangle_{\pi_{\theta}(tw)v_{\theta}}=0,$$
using the fact that $d(exp_{p})_{0}=I$, where $I$ is the identity of $T_{p}\re^{2}$, then when $t\to 0^{+}$ and we conclude that $\left\langle -\frac{w}{\left\|w\right\|},v_{\theta}\right\rangle=0$ and this is a contradiction as  $\theta\neq \tb$. 
\end{proof}
\ \\
\noindent Now we subdivide $T_{p}\re^2$ in three regions: Consider $\epsilon$ given by the Lemma \ref{L2M}, then 
\begin{eqnarray*}
R_1&=&\left\{w\in T_{p}\re^2\colon \text{the angle} \ \ \angle(w,e_1)\leq \frac{\pi}{2}-\frac{3}{2}\epsilon \ \ \text{and} \ \ \angle(w,e_1)\geq \frac{3\pi}{2}+\frac{3}{2}\epsilon\right\};\\
R_2&=&\left\{w\in T_{p}\re^2\colon \text{the angle} \ \ \frac{\pi}{4}+\frac{3}{2}\epsilon\leq\angle(w,e_1)\leq \frac{5\pi}{4}-\frac{3}{2}\epsilon\right\};\\
R_3&=&\left\{w\in T_{p}\re^2\colon \text{the angle} \ \ \frac{3\pi}{4}+\frac{3}{2}\epsilon\leq\angle(w,e_1)\leq \frac{7\pi}{4}-\frac{3}{2}\epsilon\right\}.
\end{eqnarray*}
\ \\
\noindent For $w\in T_p\re^2$, putting $a_{w}^{\bot}=\tb-\dfrac{\epsilon}{2}$ and $\tilde{a}_w^{\bot}=\tb+\dfrac{\epsilon}{2}$, where $\epsilon$ is given in Lemma \ref{L2M}.

\begin{Le}\label{L4M} For the function $\pi_{\theta}(w)$ we have that 
\begin{enumerate}
	\item There is $C_1>0$ such that for all  $w\in R_1$,
	\begin{itemize}
		\item[$\left(\mathrm{a}\right)$]  If $\left\|w\right\|\leq 1$, then    
	$\pi_{\theta}(w)\geq C_1\left\|w\right\| \ \ \text{for} \ \ \theta\in [0, a_w^{\bot}]\cup [\tilde{a}_w^{\bot},\pi].$
	  \item[$\left(\mathrm{b}\right)$] If $\left\|w\right\|\geq 1$, then 
		$\pi_{\theta}(w)\geq C_{1} \ \ \text{for} \ \ \theta\in [0, a_w^{\bot}]\cup [\tilde{a}_w^{\bot},\pi].$
	\end{itemize}
	\item There is $C_2>0$ such that for all  $w\in R_2$,
	\begin{itemize}
	\item[$\left(\mathrm{a}\right)$]  If $\left\|w\right\|\leq 1$, then  
	$\pi_{\theta}(w)\geq C_2\left\|w\right\| \ \ \text{for} \ \ \theta\in \left[\frac{3}{4}\pi, a_w^{\bot}\right]\cup \left[\tilde{a}_w^{\bot},\frac{7}{4}\pi\right].$
	\item[$\left(\mathrm{b}\right)$] If $\left\|w\right\|\geq 1$, then 
			$\pi_{\theta}(w)\geq C_{2} \ \ \text{for} \ \ \theta\in \left[\frac{3}{4}\pi, a_w^{\bot}\right]\cup \left[\tilde{a}_w^{\bot},\frac{7}{4}\pi\right].$
	\end{itemize}
	\item There is $C_3>0$ such that for all  $w\in R_3$,
	\begin{itemize}
	\item[$\left(\mathrm{a}\right)$]  If $\left\|w\right\|\leq 1$, then  
	$\pi_{\theta}(w)\geq C_3\left\|w\right\| \ \ \text{for} \ \ \theta\in \left[\frac{5}{4}\pi, a_w^{\bot}\right]\cup \left[\tilde{a}_w^{\bot},\frac{9}{4}\pi\right].$
	\item[$\left(\mathrm{b}\right)$] If $\left\|w\right\|\geq 1$, then 
			$\pi_{\theta}(w)\geq C_{3} \ \ \text{for} \ \ \theta\in \left[\frac{5}{4}\pi, a_w^{\bot}\right]\cup \left[\tilde{a}_w^{\bot},\frac{9}{4}\pi\right].$
	\end{itemize}
	
\end{enumerate}
\end{Le}
\ \\
\noindent We prove the part $1$, the parts $2$ and $3$ are analogous.
\begin{proof}[\bf{proof}]

It suffices to prove that there is $C_1>0$ such that for all  $w\in R_1$ with $\left\|w\right\|=1$ and all $t\in\left[0,1\right]$ we have 
\begin{equation}\label{E6'}
\pi_{\theta}(tw)\geq C_1t \ \ \text{for} \ \ \theta\in [0, a_w^{\bot}]\cup [\tilde{a}_w^{\bot},\pi].
\end{equation}
In fact: By contradiction, suppose that for all $n\in\mathbb{N}$ there is $w_n$ with $\left\|w_n\right\|=1$, $t_n\in [0,1]$ and $\theta_n\in [0, a_{w_n}^{\bot}]\cup [\tilde{a}_{w_{n}}^{\bot},\pi]$ such that $\pi_{\theta_n}(t_nw_n)<\frac{1}{n}t_n$.
We can assume that $w_n\to w$, $\theta_n\to \theta \in [0, a_w^{\bot}]\cup [\tilde{a}_w^{\bot},\pi]$ and $t_n\to t$ when $n\to \infty$.
If $t\neq 0$, then since for $w\in R_1$ and $\theta\in[0, a_w^{\bot}]\cup [\tilde{a}_w^{\bot},\pi]$, $\pi_{\theta}(tw)\geq 0$, we have 
$0\leq \pi_{\theta}(tw)\leq 0$, so $\theta=\tb$ and this is a contradiction, because $\theta\in [0, a_w^{\bot}]\cup [\tilde{a}_w^{\bot},\pi]$ and $\epsilon$ is fixed.\\
If $t=0$, consider the $C^{1}$-function $F(\theta, t,w)=\pi_{\theta}(tw)$, then 
$$0=\lim_{n\to \infty}\frac{F(\theta_n,t_n,w_n)}{t_n}=\lim_{t\to 0}\frac{F(\theta,t,w)}{t},$$

\noindent by Lemma \ref{L3M} we know that $\ds\lim_{t\to 0}\frac{F(\theta,t,w)}{t}\neq 0$, and this is a contradiction with the above, so the affirmation is proved.\\
\ \\
Now, since $\tb=\theta_{tw}^{\bot}$ for $t>0$ we have \\
\textbf{($\mathrm{a}$)} If $\left\|w\right\|\leq1$, then by (\ref{E6'}), $\pi_{\theta}(w)=\pi_{\theta}(\left\|w\right\|\dfrac{w}{\left\|w\right\|})\geq C_1\left\|w\right\|$ for $\theta\in [0, a_w^{\bot}]\cup [\tilde{a}_w^{\bot},\pi]$.\\
\ \\
\textbf{($\mathrm{b}$)} Since $\pi_{\theta}(w)\geq \pi_{\theta}(\frac{w}{\left\|w\right\|})$ for $\left\|w\right\|\geq 1$,  then the equation (\ref{E6'}) and  implies the result.
\end{proof}

\subsection{The Bessel Function Associated to $\pi_{\theta}(w)$}
For $w\in T_{p}\re^{2}$ consider the Bessel function 
$$\tilde{J}_{w}(z)=\int_{0}^{2\pi}\cos(z\pi_{\theta}(w))d\theta.$$
Observe that we can consider $\pi_{\theta}(w)$ as a periodic function in $\theta$ of period  $2\pi$. Moreover, $\tilde{J}_{w}(z)$ has the following properties:
\begin{enumerate}
	\item $\tilde{J}_{w}(z)=\tilde{J}_{w}(-z)$;
	\item $\ds\tilde{J}_{w}(z)=\int_{0}^{2\pi}\cos(z\pi_{\theta}(w))d\theta=\int_{t}^{2\pi+t}\cos(z\pi_{\theta}(w))d\theta$ for any $t\in \re$.
	\item As $\pi_{\theta+\pi}(\exp_p(w))=-\pi_\theta(\exp_p(w))$, then
\begin{eqnarray*}
\ds\int^{\pi+t}_{t}\cos(z\pi_{\theta}w) d\theta&=&\int^{2\pi+t}_{\pi+t}\cos(z\pi_{\theta-\pi}w)d\theta=\int^{2\pi+t}_{\pi+t}\cos(-z\pi_{\theta}w)d\theta\\
&=&\int^{2\pi+t}_{\pi+t}\cos(z\pi_{\theta}(w)d\theta,
\end{eqnarray*}
\end{enumerate}
\noindent Thus, 

\begin{equation}\label{E7}
\tilde{J}_w(z)=\ds2\int^{\pi+t}_{t}\cos(z\pi_\theta(w))d\theta:=2J^{t}_w(z).
\end{equation}
\ \\
\begin{R}\label{R1M}
To fix ideas we consider 
\begin{eqnarray*}
t&=&0 \ \ \text{for} \ \ w\in R_1;\\
t&=&\frac{3}{4}\pi \ \ \text{for} \ \ w\in R_2;\\
t&=&\frac{5}{4}\pi \ \ \text{for} \ \  w\in R_3.\\
\end{eqnarray*}
\end{R}

\begin{Pro}\label{P1M} For any $w\in T_{p}\re^2$ we have that 
$\ds\int^{\infty}_{-\infty}\tilde{J}_{w}(z)dz<\infty$.
\end{Pro}
\begin{proof}[\bf{proof}]
We divide the proof in three parts.
\begin{enumerate}
	\item If $w\in R_1$, in this case, by Remark \ref{R1M} and equation (\ref{E7}) is it suffices to prove the Lemma for $J_{w}^{0}(z):=J_{w}(z)$.
   
	\item If $w\in R_2$, in this case, by Remark \ref{R1M} and equation (\ref{E7}) is it suffices to prove the Lemma for $J_{w}^{{3\pi}/{4}}(z)$.
	
	\item If $w\in R_3$, in this case, by Remark \ref{R1M} and equation (\ref{E7}) is it suffices to prove the Lemma for $J_{w}^{{5\pi}/{4}}(z)$.
\end{enumerate}
We will prove 1, the proof of 2 and 3 are analogous. In fact: Since $J_{w}(z)=J_{w}(-z)$, then 
$$\ds\int^{\infty}_{-\infty}{J}_{w}(z)dz=2\ds\int^{\infty}_{0}{J}_{w}(z)dz,$$
so, the proof is reduced to prove that $\ds\int^{\infty}_{0}{J}_{w}(z)dz<\infty.$\\
Let $w\in R_1$ and $x>0$, then 
\begin{eqnarray*}
\ds\int^{x}_{0}{J}_w(z)dz&=&\int^{\pi}_{0}\int^{x}_{0}\cos(z\pi_\theta(w))dzd\theta=\int^{\pi}_{0}\frac{\sin(x\pi_\theta(w))}{\pi_\theta(w)}d\theta\\
&=&\int^{\tb}_{0}\frac{\sin(x\pi_\theta(w))}{\pi_\theta(w)}d\theta+\int^{\pi}_{\tb}\frac{\sin(x\pi_\theta(w))}{\pi_\theta(w)}d\theta:=I^{1}_{w}(x)+I^{2}_{w}(x).
\end{eqnarray*}
The next step is to estimate $I^{1}_{w}(x)$ and $I^{2}_{w}(x)$.\\
\begin{equation}\label{E8}
\ds I^{1}_{w}(x)=\int^{a_{w}^{\bot}}_{0}\frac{\sin(x\pi_\theta(w))}{\pi_\theta(w)}d\theta+\int^{\tb}_{a_{w}^{\bot}}\frac{\sin(x\pi_\theta(w))}{\pi_\theta(w)}d\theta,
\end{equation}
where $a^{\bot}_{w}=\tb-\epsilon$.

\noindent Now, by Lemma \ref{L4M}.1 we have that for $\theta\in [0,a_w^{\bot}]$ and $\left\|w\right\|\leq 1$, then  $\pi_{\theta}(w)\geq C_1\left\|w\right\|$ and for $\left\|w\right\|\geq 1$, $\pi_{\theta}(w)\geq C_1$.\\ 
\noindent Since, $\sin(x\pi_\theta(w))\leq 1$, then the first integral on the right side  (\ref{E8}) is bounded in $x$. In fact:
\begin{equation}\label{E9}
 \int^{a_{w}^{\bot}}_{0}\frac{\sin(x\pi_\theta(w))}{\pi_\theta(w)}d\theta \leq \left\{ \begin{array}{lll}
         \frac{\pi}{C_1\left\|w\right\|}& \mbox{\text{if}\ \ $0<\left\|w\right\|\leq 1$};\\
          & \\ 
       \frac{\pi}{C_1}& \mbox{\text{if} \ \ $\left\|w\right\|>1$. \ 
        }\end{array} \right.
\end{equation} 
					
\noindent 	Now we estimate the second integral on the right side of (\ref{E8}).\\
\ \\
Put $f_w(\theta)=\pi_\theta(w)$, then $f_{w}(\tb)=0$ and $f_{w}(\theta)>0$ for $\theta<\tb$. Moreover, recall that by Lemma \ref{L2M},  $\frac{\partial \pi}{\partial \theta}(\tb,w)=-\left\|w\right\|\neq 0$, then $f_w'(\tb)\neq 0$, and put $s=f_{w}(\theta)$. Thus,

\begin{equation}\label{E10}
\ds\int^{\tb}_{a_w^{\bot}}\frac{\sin(x\pi_\theta(w))}{\pi_\theta(w)}d\theta=-\int^{f_{w}(a_w^{\bot})}_{0}\frac{\sin \left(xs\right)}{sf'_{w}(f_w^{-1}(s))}ds=
-\int^{f_w(a_{w}^{\bot})}_{0}\frac{\sin \left(xs\right)}{sg_w(s)}ds,
\end{equation}
where $g_w(s)=f'_w(f_w^{-1}(s))$ is  $C^\infty$.\\
\ \\
Now by definition of $s$, if $s\in [0,f_{w}(a_{w}^{\bot})]$, then $f_{w}^{-1}(s)\in [a_w^{\bot},\tb]$. Thus by Lemma \ref{L2M} we have 

\begin{equation}\label{E11}
-\left\|w\right\|\leq g_w(s)\leq-\frac{1}{2}\left\|w\right\| \ \ \text{for all} \ \ s\in[0,f_{w}(a_{w}^{\bot})].
\end{equation}
For large $x$ 
\begin{center}
$\ds-\int^{\frac{2\pi}{x}}_{0}\frac{\sin \left(xs\right)}{sg_\alpha(s)}ds=-\int^{\frac{\pi}{x}}_{0}\frac{\sin \left(xs\right)}{sg_\alpha(s)}ds-\int^{2\pi/x}_{\pi/x}\frac{\sin \left(xs\right)}{sg_\alpha(s)}ds$
\end{center}
Since  $\sin\left(xs\right)\geq0$ in $\left[0,\frac{\pi}{x}\right]$ then  
\begin{equation}\label{E12}
\ds-\int^{\frac{\pi}{x}}_{0}\frac{\sin \left(xs\right)}{sg_\alpha(s)}ds\leq\frac{2}{\left\|w\right\|}\int^{\frac{\pi}{x}}_{0}\frac{\sin \left(xs\right)}{s}ds=\frac{2}{\left\|w\right\|}\int^{\pi}_{0}\frac{\sin y}{y}dy.
\end{equation}

\noindent As well $-sin\left(xs\right)\geq0$ for $s\in\left[\frac{\pi}{x},\frac{2\pi}{x}\right]$, then $\ds-\int^{2\pi/x}_{\pi/x}\frac{\sin\left(xs\right)}{sg_w(s)}ds\leq 0$. So, by (\ref{E12})
\begin{equation}\label{E13}
\ds-\int^{\frac{2\pi}{x}}_{0}\frac{\sin \left(xs\right)}{sg_w(s)}ds\leq \frac{2}{\left\|w\right\|}\int^{\pi}_{0}\frac{\sin y}{y}dy.
\end{equation}
\ \\
Let $n\in \mathbb{N}$ such that $\ds n\leq \frac{xf_w(a_w^{\bot})}{2\pi}\leq n+1$, then 
\begin{center}
$\ds\int^{f_w(a_w^{\bot})}_{0}\frac{\sin \left(xs\right)}{sg_w(s)}ds=\int^{\frac{2\pi}{x}}_{0}\frac{\sin \left(xs\right)}{sg_w(s)}ds+\sum^{k=n-1}_{k=1}\int^{\frac{2\pi(k+1)}{x}}_{\frac{2\pi k}{x}}\frac{\sin \left(xs\right)}{sg_w(s)}ds+\int^{f_w(a_w^{\bot})}_{\frac{2\pi n}{x}}\frac{\sin \left(xs\right)}{sg_w(s)}ds.$
\end{center}
If $\frac{2\pi n}{x}\leq f_w(a_w^{\bot})\leq \frac{\pi(2n+1)}{x}$, then $\sin\left(xs\right)\geq 0$ and by Lemma \ref{L2M}, we have  
\ \\
$$\ds \frac{\sin\left(xs\right)}{s\left\|w\right\|}\leq -\frac{\sin \left(xs\right)}{sg_w(s)}\leq \frac{2\sin\left(xs\right)}{s\left\|w\right\|} \ \ \text{and} \ \ \ds \frac{2\sin\left(xs\right)}{s\left\|w\right\|} \leq \frac{2x \sin \left(xs\right)}{\left\|w\right\|2\pi n}.$$
This implies 
\begin{eqnarray*}
\ds-\int^{f_w(a_w^{\bot})}_{\frac{2\pi n }{x}}\frac{\sin \left(xs\right)}{sg_w(s)}ds&\leq& \int^{f_w(a_w^{\bot})}_{\frac{2\pi n}{x}}\frac{x \sin \left(xs\right)}{\left\|w\right\|\pi n}ds\leq\frac{x}{\left\|w\right\|\pi n}\int^{f_w(a_w^{\bot})}_{\frac{2\pi n}{x}}{\sin \left(xs\right)}ds\\
&\leq& \frac{x}{\left\|w\right\| \pi n}\left(f_w\left(a_w^{\bot}\right)-\frac{2\pi n}{x}\right)=\frac{2}{\left\|w\right\|}\left(\frac{xf_w(a_w^{\bot})}{2\pi n}-1\right) \\
&\leq& \frac{2}{\left\|w\right\|}\left(\frac{2\pi(n+1)}{2\pi n}-1\right)
=\frac{2}{\left\|w\right\|}\frac{1}{n}.
\end{eqnarray*}
In the case that $f_w(\ab)\geq \frac{\pi(2n+1)}{x}$, then $\ds-\int^{f_w(\ab)}_{\frac{\pi(2n+1)}{x}}\frac{\sin \left(xs\right)}{sg_w(s)}ds\leq 0$, so 
\begin{eqnarray*}
\ds-\int^{f_w(\ab)}_{\frac{2\pi n}{x}}\frac{\sin \left(xs\right)}{sg_w(s)}ds&\leq& -\int^{\frac{\pi(2n+1)}{x}}_{\frac{\pi (2n+1)}{x}}\frac{\sin \left(xs\right)}{sg_w(s)}ds\leq \frac{x}{\left\|w\right\|2\pi n }\left(\frac{\pi(2n+1)}{x}-\frac{2\pi n}{x}\right)\\
&=&\frac{1}{\left\|w\right\|}\frac{1}{n}.
\end{eqnarray*}
In any case, we have 
\begin{equation}\label{E14}
\ds\int^{f_w(\ab)}_{\frac{2\pi n}{x}}\frac{\sin \left(xs\right)}{sg_w(s)}ds \leq \frac{2}{\left\|w\right\|}\frac{1}{n}.
\end{equation}
\ \\
Now we only need to estimate $\ds\sum^{k=n-1}_{k=1}\int^{\frac{2\pi(k+1)}{x}}_{\frac{2\pi k}{x}}\frac{\sin \left(xs\right)}{sg_w(s)}ds$.\\
Put $s_0=\frac{2\pi k}{x}$, then 
\begin{center}
$\ds\int^{\frac{2\pi (k+1)}{x}}_{\frac{2\pi k}{x}}\frac{\sin\left(xs\right)}{sg_w(s)}ds=\int^{\frac{2\pi (k+1)}{x}}_{\frac{2\pi k}{x}}\frac{\sin\left(xs\right)}{s_0g_w(s_0)}ds+\int^{\frac{2\pi (k+1)}{x}}_{\frac{2\pi k}{x}} \sin \left(xs\right) \left(\frac{1}{sg_w(s)}-\frac{1}{s_0g_w(s_0)}\right)ds.$
\end{center}
\ \\
The first integral on the right term of the above equality is zero.\\
\ \\
Now we estimate the second integral on the right side of the above equation.\\
\ \\
By Lemma \ref{L2M} we have $g_w(s)g_w(s_0)>\dfrac{\left\|w\right\|^{2}}{4}$, also $ss_0\geq \left(\frac{2\pi k}{x}\right)^2$.\\
Thus, $\ds \frac{1}{ss_0g_w(s)g_w(s_0)}<\frac{1}{\left\|w\right\|^2}\frac{x^2}{\pi^2k^2}$. Moreover,
\ \\
\begin{eqnarray*}
\left|s_0g_w(s_0)-sg_w(s)\right|&=&\left|(s_0-s)g_w(s_0)+s(g_w(s_0)-g_w(s))\right|\\
&\leq&\left|s_0-s\right|\left|g_w(s_0)\right|+s\left|g_w(s)-g_w(s_0)\right|\\ 
&\leq& \left|s-s_0\right|\left(\left|g_w(s_0)\right|+s \sup_{s\in [0,f_w(\ab)]}\left|g'_w(s)\right|\right) \ \ \downarrow \ \ \text{by Lemma \ref{L2M}} \ \ \\
&\leq& \frac{2\pi}{x}\left(\left\|w\right\|+\frac{2\pi(k+1)}{x}\left\|w\right\|\right)\\
&\leq& \frac{2\pi}{x}\left\|w\right\|\left(1+\frac{2\pi n}{x}\right)\\
&\leq& \frac{2\pi}{x}\left\|w\right\|\left(1+f_{w}(\ab)\right) \\
&\leq& \frac{2\pi}{x}\left\|w\right\|\left(1+\left\|w\right\|\right)
\end{eqnarray*}
as, $f_{w}(\ab)\leq \left\|w\right\|$. 
Therefore, 
\begin{center}
$\ds \left|\frac{1}{sg_w(s)}-\frac{1}{s_0g_w(s_0)}\right|=\left|\frac{sg_w(s)-s_0g_w(s_0)}{ss_0g_w(s)g_w(s_0)}\right|\leq \frac{2\left(1+\left\|w\right\|\right)}{\pi\left\|w\right\|}\left(\frac{x}{k^2}\right).$
\end{center}
Then, 
\begin{center}
$\ds\left|\int^{\frac{2\pi (k+1)}{x}}_{\frac{2\pi k}{x}}\frac{\sin\left(xs\right)}{sg_w(s)}ds\right|\leq \frac{2\left(1+\left\|w\right\|\right)}{\pi\left\|w\right\|}\left(\frac{x}{k^2}\right)\left(\frac{2\pi(k+1)}{x}-\frac{2\pi k}{x}\right)=\frac{4\left(1+\left\|w\right\|\right)}{\left\|w\right\|}\left(\frac{1}{k^2}\right)$.
\end{center}
Therefore,
\begin{equation}\label{E15}
\ds\left|\sum^{k=n-1}_{k=1}\int^{\frac{2\pi (k+1)}{x}}_{\frac{2\pi k}{x}}\frac{\sin\left(xs\right)}{sg_w(s)}ds\right|\leq \sum^{k=n-1}_{k=1}\left|\int^{\frac{2\pi (k+1)}{x}}_{\frac{2\pi k}{x}}\frac{\sin\left(xs\right)}{sg_w(s)}ds\right|\leq {A(\left\|w\right\|)}\sum^{k=n-1}_{k=1}\frac{1}{k^2},
\end{equation}
where $A(\left\|w\right\|)=\dfrac{4\left(1+\left\|w\right\|\right)}{\left\|w\right\|}$.\\
\ \\
Since $\sum^{\infty}_{k=1}\frac{1}{k^2}:=a<\infty$ and put $b=\int^{\pi}_{0}\frac{\sin y}{y}$, then the equations (\ref{E13}), (\ref{E14}), and (\ref{E15}) imply
\begin{equation}\label{E16}
-\int^{f_w(\ab)}_{0}\frac{\sin(xs)}{sg_w(s)}ds\leq \frac{2}{\left\|w\right\|}b+\frac{2}{\left\|w\right\|}\frac{1}{n}+A(\left\|w\right\|)a.
\end{equation}

\noindent Thus, by the equation (\ref{E8}), (\ref{E9}), (\ref{E10}), and (\ref{E16}) we have

\begin{equation}\label{E17}
  I^{1}_{w}(x) \leq \left\{ \begin{array}{lll}
         \ds\frac{\pi}{C_1\left\|w\right\|}+\frac{2}{\left\|w\right\|}b+\frac{2}{\left\|w\right\|}\frac{1}{n}+A(\left\|w\right\|)a & \mbox{\text{if}\ \ $0<\left\|w\right\|\leq 1$};\\
          & \\ 
       \ds\frac{\pi}{C_1}+\frac{2}{\left\|w\right\|}b+\frac{2}{\left\|w\right\|}\frac{1}{n}+A(\left\|w\right\|)a& \mbox{\text{if} \ \ $\left\|w\right\|>1$. \ 
        }\end{array} \right. 
\end{equation} 

\noindent Completely analogous using $\tilde{a}_{w}^{\bot}$ instead of $\ab$ and  taking $n'$ such that \\
$-\frac{\pi(2n'+1)}{x}\leq f_w(\tilde{a}_w^{\bot})\leq -\frac{2\pi n'}{x}$, we also obtain
				
\begin{equation}\label{E18}
  I^{2}_{w}(x) \leq \left\{ \begin{array}{lll}
         \ds\frac{\pi}{C_1\left\|w\right\|}+\frac{2}{\left\|w\right\|}b+\frac{2}{\left\|w\right\|}\frac{1}{n'}+A(\left\|w\right\|)a & \mbox{\text{if}\ \ $0<\left\|w\right\|\leq 1$};\\
          & \\ 
       \ds\frac{\pi}{C_1}+\frac{2}{\left\|w\right\|}b+\frac{2}{\left\|w\right\|}\frac{1}{n'}+A(\left\|w\right\|)a& \mbox{\text{if} \ \ $\left\|w\right\|>1$. \ 
        }\end{array} \right. 
\end{equation} 
Since $n,n'\to \infty$ as $x\to\infty$, then (\ref{E17}) and (\ref{E18}) implies  

\begin{equation}\label{E19}
  \int^{\infty}_{0}J_w(z)dz\leq \left\{ \begin{array}{lll}
         \ds\frac{2\pi}{C_1\left\|w\right\|}+\frac{4}{\left\|w\right\|}b+2A(\left\|w\right\|)a & \mbox{\text{if}\ \ $0<\left\|w\right\|\leq 1$};\\
          & \\ 
       \ds\frac{2\pi}{C_1}+\frac{2}{\left\|w\right\|}b+2A(\left\|w\right\|)a& \mbox{\text{if} \ \ $\left\|w\right\|>1$. \ 
        }\end{array} \right. 
\end{equation} 
\noindent Thus, we conclude the proof of Proposition \ref{P1M}.
\end{proof}

\noindent Put $j_1=0$, $j_2=\dfrac{3\pi}{4}$ and $j_3=\dfrac{5\pi}{4}$, then it is also easy to see that for $w\in R_i$, 
\begin{equation}\label{E20}
  \int^{\infty}_{0}J^{j_i}_w(z)dz\leq \left\{ \begin{array}{lll}
         \ds\frac{2\pi}{C_i\left\|w\right\|}+\frac{4}{\left\|w\right\|}b+2A(\left\|w\right\|)a & \mbox{\text{if}\ \ $0<\left\|w\right\|\leq 1$};\\
          & \\ 
       \ds\frac{2\pi}{C_i}+\frac{2}{\left\|w\right\|}b+2A(\left\|w\right\|)a& \mbox{\text{if} \ \ $\left\|w\right\|>1$, \ 
        }\end{array} \right. 
	\end{equation}			
$i=1,2,3$, where $C_i$ are given in Lemma \ref{L4M}.
\section{Proof of the Main Theorem.}
\noindent As in the Kaufman's proof of Marstrand's theorem  (cf. \cite{Kaufman}), we use the potential theory.\\
Put $d=HD(K)>1$, assume that $0<M_d(K)<\infty$ and for some $C>0$, we have 
\begin{center}
$m_d(K\cap B_r(x))\leq Cr^d$
\end{center}
for $x\in \mathbb{R}^2$ and $0<r\leq 1$ (cf. \cite{Falconer}). Let $\mu$ be the finite measure on $\mathbb{R}^2$ defined by $\mu(A)=m_d(K \cap A)$, $A$ a measurable subset of $\mathbb{R}^2$. For $-\frac{\pi}{2}<\theta<\frac{\pi}{2}$, let us denote by $\mu_\theta$ the (unique) measure on $\mathbb{R}$ such that $\int {fd\mu_\theta}=\int{(f\circ\pi_\theta)}d\mu$ for every continuous function $f$. The theorem will follow, if we show that the support of $\mu_\theta$ has positive Lebesgue measure for almost all $\theta\in(-\frac{\pi}{2},\frac{\pi}{2})$, since this support is clearly contained in $\pi_\theta(K)$. To do this we use the following fact.\\
\begin{Le}\label{L5M}$\left(\emph{cf. \cite[pg. 65]{PT}}\right)$
Let $\eta$ be a finite measure with compact support on $\mathbb{R}$ and $$\hat{\eta}(p)=\frac{1}{\sqrt{2\pi}}\int^{\infty}_{-\infty}e^{-ixp}d\eta(x),$$ for $p\in\mathbb{R}$ $\left(\hat{\eta} \ \ \text{is the fourier transform of} \ \ \eta\right)$. If $0<\int^{\infty}_{-\infty}|\hat{\eta}(p)|^2dp<\infty$ then the support of $\eta$ has positive Lebesgue measure.
\end{Le}
\begin{proof}[\bf{Proof of the Main Theorem}] \ \\
We now show that, for almost any $\theta\in\left(-\frac{\pi}{2},\frac{\pi}{2}\right)$, $\hat{\mu}_\theta$ is square-integrable. From the definitions we have 
\begin{center}
$|\hat{\mu}_\theta(p)|^2=\ds\frac{1}{2\pi}\int \int e^{i(y-x)p}d\mu_\theta(x)d\mu_\theta(y)=\frac{1}{2\pi}\int\int e^{ip(\pi_\theta(v)-\pi_\theta(u))}d\mu(u)d\mu(v)$
\end{center}
as $\pi_{\theta+\pi}(u)=-\pi_\theta(u)$, then
\begin{center}
$|\hat{\mu}_\theta(p)|^2+|\hat{\mu}_{\theta+\pi}(p)|^2=\ds \frac{1}{\pi}\int\int \cos(p(\pi_\theta(v)-\pi_\theta(u)))d\mu(u)d\mu(v)$.
\end{center}
And so 
\begin{eqnarray*}
\ds\int^{2\pi}_{0}|\hat{\mu}_\theta(p)|^2d\theta&=&\frac{1}{2\pi}\int^{2\pi}_{0}\int\int cos(p(\pi_\theta(v)-\pi_\theta(u)))d\mu(u)d\mu(v)d\theta\\
&=&\frac{1}{2\pi}\int\int\left(\int^{2\pi}_{0} cos(p(\pi_\theta(v)-\pi_\theta(u)))d\theta \right)d\mu(u)d\mu(v).
\end{eqnarray*}
Observe now that for all $x>0$ and for all $u,v$ there are $L\in \mathbb{N}$ and $w(u,v)$ such that
\begin{center}
$\ds\int^{x}_{0}\int^{2\pi}_{0}\cos(p(\pi_{\theta}(u)-\pi_{\theta}(v)))d\theta dp 
\leq L \left|\int^{x}_{0}\int^{2\pi}_{0}\cos(p\pi_{\theta}(w(u,v))d\theta dp\right|$
\end{center}
$w(u,v)$ can be taken such that $d(p,w)=d(u,v)$.
So, we have for $x>0$
\begin{center}
$\ds \int^{x}_{-x}\int^{2\pi}_{0}|\hat{\mu}_\theta(p)|^2d\theta dp \leq \frac{2L}{2\pi}\int\int\left|\int^{x}_{0}\tilde{J}_{w(u,v)}(p)dp\right|d\mu(u)d\mu(v)$.
\end{center}
Follows
\begin{equation*}
\ds \frac{\pi}{L}\int^{\infty}_{-\infty}\int^{2\pi}_{0}|\hat{\mu}_\theta(p)|^2d\theta dp\leq\int\int\left|\int^{\infty}_{0}\tilde{J}_{w(u,v)}(p)dp\right|d\mu(u)d\mu(v)=
\end{equation*}
\begin{equation*}
 =\int\int_{\{\left\|w\right\|>1\}}\left|\int^{\infty}_{0}\tilde{J}_{w(u,v)}(p)dp\right|d\mu(u)d\mu(v)+\int \int_{\left\{\left\|w\right\|\leq1\right\}}\left|\int^{\infty}_{0}\tilde{J}_{w(u,v)}(p)dp\right|d\mu(u)d\mu(v)\\
\end{equation*}
\begin{equation}\label{E21}
=:I+II.
\end{equation}
By (\ref{E7}) and Remark \ref{R1M}
\begin{eqnarray*}
I&=&{\int\int}_{\left\{\left\|w\right\|>1\right\}}\left|\int^{\infty}_{0}\tilde{J}_{w}(p)dp\right|d\mu(u)d\mu(v)=\sum_{i=1}^{3} \int\int_{\{\left\|w\right\|>1\}\cap R_i}\left|\int^{\infty}_{0}\tilde{J}_{w}(p)dp\right|d\mu(u)d\mu(v)\\
&=&2\sum_{i=1}^{3} \int\int_{\{\left\|w\right\|>1\}\cap R_i}\left|\int^{\infty}_{0}{J}^{j_i}_{w}(p)dp\right|d\mu(u)d\mu(v).
 \end{eqnarray*}
Now by (\ref{E19}) and (\ref{E20}), we have 
\begin{eqnarray*} 
I\leq 2\sum_{i=1}^{3}\ds\int\int_{\{\left\|w\right\|>1\}\cap R_i}\left(\frac{2\pi}{C_i}+\frac{2}{\left\|w\right\|}b+2A(\left\|w\right\|)a\right)d\mu(u)d\mu(v).
\end{eqnarray*}
If $\left\|w\right\|>1$, then $\frac{1}{\left\|w\right\|}<1$ and $A(\left\|w\right\|)=\frac{4(1+\left\|w\right\|)}{\left\|w\right\|}<8$, moreover, as the support of the measure $\mu\times \mu$ is contained in $K \times K$ which is compact, then 
\begin{equation}\label{E22}
I\leq 6\left(2\pi\max\left\{\frac{1}{C_i}\right\}+2b+16a\right)\mu(K)^2.
\end{equation}
We now estimate $II$, in fact: By (\ref{E7}) and Remark \ref{R1M},
\begin{eqnarray*}
II&=&\int \int_{\left\{\left\|w\right\|\leq1\right\}}\left|\int^{\infty}_{0}\tilde{J}_{w}(p)dp\right|d\mu(u)d\mu(v)=\sum_{i=1}^{3}\int \int_{\left\{\left\|w\right\|\leq1\right\}\cap R_i}\left|\int^{\infty}_{0}\tilde{J}_{w}(p)dp\right|d\mu(u)d\mu(v)\\
&=&2\sum_{i=1}^{3} \int\int_{\{\left\|w\right\|\leq1\}\cap R_i}\left|\int^{\infty}_{0}{J}^{j_i}_{w}(p)dp\right|d\mu(u)d\mu(v).
\end{eqnarray*}
Now by (\ref{E19}) and (\ref{E20}), we have 
\begin{eqnarray}\label{E23}
II&\leq& 2\sum_{i=1}^{3} \int\int_{\{\left\|w\right\|\leq1\}\cap R_i}\left(\frac{2\pi}{C_i\left\|w\right\|}+\frac{4}{\left\|w\right\|}b+2A(\left\|w\right\|)a\right)d\mu(u)d\mu(v) \nonumber \\
&\leq&6\int\int_{\{\left\|w\right\|\leq1\}}\left(\left(\max\left\{\frac{2\pi}{C_i}\right\}+4b+8a\right)\frac{1}{\left\|w\right\|} +8a\right)d\mu(u)d\mu(v).
\end{eqnarray}
\noindent Remember that $\left\|w(u,v)\right\|=d(u,v)$, then
\begin{center}
$\ds\int\int_{\left\{\left\|w\right\|\leq 1\right\}}\frac{1}{\left\|w\right\|}d\mu(u)d\mu(v)=\int\int_{\left\{d(u,v)\leq 1\right\}}\frac{1}{d(u,v)}d\mu(u)d\mu(v)$.
\end{center}
Now, for some  $0<\beta<1$
\begin{eqnarray*}
\ds\int_{\left\{\left\|w\right\|\leq 1\right\}}\frac{1}{d(u,v)}d\mu(v)&=&\sum^{\infty}_{n=1}\int_{{\beta}^n\leq d(u,v)\leq {\beta}^{n-1}}\frac{d\mu(v)}{d(u,v)}\leq \sum^{\infty}_{n=1}{\beta}^{-n}\mu(B_{{\beta}^{n-1}}(u))\\
&\leq& C\sum^{\infty}_{n=1}{\beta}^{-n}({\beta}^{n-1})^d \ \ \   \\ 
&\leq& C\ds\sum^{\infty}_{n=1}{\beta}^{-d}({\beta}^{d-1})^n \ \text{with} \ \ d>1\\
&=&C{\beta}^{-d}\left(\frac{1}{1-{\beta}^{d-1}}-1\right)=\frac{C}{\beta-{\beta}^{d}}.
\end{eqnarray*}

\noindent Therefore, 
\begin{center}
$\ds\int\int_{\left\{\left\|w\right\|\leq 1\right\}}\frac{1}{\left\|w\right\|}d\mu(u)d\mu(v)\leq \mu({\mathbb{R}^2})\frac{C}{\beta-{\beta}^d}$.
\end{center}

\noindent Also, $\ds\int\int_{\left\{\left\|w\right\|\leq 1\right\}} 48d\mu(u)d\mu(v)\leq 8a{\mu(K)}^2<\infty$.\\
Using these last two inequalities and the equation (\ref{E23}) we have that  
\begin{equation}\label{E24}
II\leq 6\left(\left(\max\left\{\frac{2\pi}{C_i}\right\}+4b+8a\right)\frac{C}{\beta-{\beta}^{d}} +8a\mu(K)^{2}\right).
\end{equation}

\noindent Using Fubini, the by equations (\ref{E21}), (\ref{E22}) and (\ref{E24}) we have 
\begin{center}
$\ds \frac{\pi}{L}\int^{2\pi}_{0}\int^{\infty}_{-\infty}|\hat{\mu}_\theta(p)|^2dpd\theta \leq I+II \leq 6\left(2\pi\max\left\{\frac{1}{C_i'}\right\}+2b+16a\right)\mu(K)^2+6\left(\left(\max\left\{\frac{2\pi}{C_i}\right\}+4b+8a\right)\frac{C}{\beta-{\beta}^{d}} +8a\mu(K)^{2}\right)<\infty$.
\end{center}
Therefore, $\ds\int^{\infty}_{-\infty}|\hat{\mu}_\theta(p)|^2dp<\infty$ for almost all $\theta \in \left(-\frac{\pi}{2},\frac{\pi}{2}\right)$.\\
If exists $\theta \in \left(-\frac{\pi}{2},\frac{\pi}{2}\right)$ such that $\ds\int^{\infty}_{-\infty}|\hat{\mu}_\theta(p)|^2dp=0$, then $\ds\int^{\infty}_{-\infty}\left|\varphi(x)\right|^2dx=\int^{\infty}_{-\infty}|\hat{\mu}_\theta(p)|^2dp=0$ where $\varphi(x)=\ds\frac{1}{\sqrt{2\pi}}\int^{\infty}_{-\infty}e^{ixp}\hat{\mu_{\theta}}(p)dp$. This implies that $\varphi\equiv 0$ almost every where, but $d\mu_{\theta}=\varphi dx$. This is $\mu_{\theta}(\mathbb{R})=\int^{\infty}_{-\infty}\varphi(x)dx=0$ and this implies that $\mu(\mathbb{R}^2)=0$, this contradicts the fact that $d$-measure of Haussdorff of $K$ is positive.\\
\ \\
The result follows of Lemma \ref{L5M}, in the case $0<m_d(K)<\infty$. \\
\ \\
In the general case, we take $0<m_{d'}(K')<\infty$ with $1<d'<d$ and $K'\subset K$ (cf. \cite{Falconer}).  Then, by the same argument $\pi_{\theta}(K')$ has positive measure for almost all $\theta$, and since $\pi_{\theta}(K')\subset \pi_{\theta}(K)$, then the same is true for $\pi_{\theta}(K)$.
\end{proof}

\newpage

$$\textbf{Acknowledgments}$$
The author is thankful to IMPA for the excellent ambient during
the preparation of this manuscript. The author is also grateful to Carlos Gustavo Moreira for carefully reading the preliminary version of this work and their comments in this work. This work was financially supported by CNPq-Brazil, Capes, and the Palis Balzan Prize.

\bibliographystyle{alpha}	
\bibliography{bibtex}
\noindent \textbf{Sergio Augusto Roma\~na Ibarra}\\
Universidade Federal do Rio de Janeiro, Campus Maca\'e\\
Av. do Alo\'izio, 50 - Gl\'oria, Maca\'e \\ 
27930-560 Rio de Janeiro-Brasil\\
E-mail: sergiori@macae.ufrj.br\\

\end{document}